\documentclass{amsart}

\usepackage{amssymb, amsmath}
\usepackage{mathrsfs}
\usepackage{amscd}
\usepackage{verbatim}

\allowdisplaybreaks

\usepackage{hyperref}

\theoremstyle{plain}
\newtheorem{theorem}{Theorem}[section]
\theoremstyle{remark}
\newtheorem{remark}[theorem]{Remark}
\newtheorem{example}[theorem]{Example}
\theoremstyle{plain}
\newtheorem{corollary}[theorem]{Corollary}
\newtheorem{lemma}[theorem]{Lemma}
\newtheorem{proposition}[theorem]{Proposition}

\numberwithin{equation}{section}

\def\R{{\mathbb R}}
\def\C{{\mathbb C}}

\newcommand{\E}{{\mathbb E}}
\renewcommand{\P}{{\mathbb P}}
\newcommand{\F}{{\mathscr F}}

\newcommand{\g}{\gamma}

\renewcommand{\O}{\Omega}

\newcommand{\calL}{{\mathscr L}}
\newcommand{\n}{\Vert}
\newcommand{\nnn}{|\!|\!|}
\newcommand{\one}{{{\bf 1}}}

\newcommand{\beq}{\begin{equation}}
\newcommand{\eeq}{\end{equation}}

\begin{document}

\author{Mark Veraar}
\address{Delft Institute of Applied Mathematics\\
Delft University of Technology \\ P.O. Box 5031\\ 2600 GA Delft\\The
Netherlands} \email{M.C.Veraar@tudelft.nl}

\author{Lutz Weis}
\address{Institut f\"ur Analysis \\
Universit\"at Karlsruhe (TH)\\
D-76128  Karls\-ruhe\\Germany}
\email{Lutz.Weis@math.uni-karlsruhe.de}

\title[Maximal estimates for stochastic convolutions]{A note on maximal estimates for stochastic convolutions}

\begin{abstract}
In stochastic partial differential equations it is important to have pathwise regularity properties of stochastic convolutions. In this note we present a new sufficient condition for the pathwise continuity of stochastic convolutions in Banach spaces.
\end{abstract}




\thanks{The first author was supported by a VENI subsidy 639.031.930
of the Netherlands Organization for Scientific Research (NWO). The second named author is supported by a grant from the Deutsche Forschungsgemeinschaft (We 2847/1-2).}

\maketitle

\section{Introduction and main result}
Let $(\O, \mathcal{A}, \P)$ be a probability space with filtration $(\F_t)_{t\geq 0}$.
Let $(S(t))_{t\geq 0}$ be a strongly continuous semigroup on a Banach space $X$. We will be interested in obtaining conditions for path-continuity of the stochastic convolution
\[S\diamond G(t):=\int_0^t S(t-s) G(s)\, d W_H(s),\]
where $G:\R_+\times\O\to \calL(H,X)$ is such that the stochastic integral with respect to the cylindrical Brownian motion $W_H$ exists. There are many such continuity results in the literature (see \cite{Brz2,BrzPes, DPZ, HauSei1, HauSei2, NVW3} and references therein).

Our methods to obtain continuity results are based on techniques similar to the ones in \cite{BrzPes,HauSei1,HauSei2}. The results are comparable with \cite{BrzPes} but are of independent interest.
The methods we present can also be applied for other stochastic convolutions $\int_0^t S(t-s) d M(s)$, where $M$ is an $X$-valued local martingale, as soon as one has a ``decent" stochastic integration theory for integration with respect to $M$. For instance, some of our methods can also be applied in the case $W_H$ is replaced by a continuous local martingale (see \cite{Vcont}) or a L\'evy process (cf. \cite{BrzHaus}). In this paper this has not been considered and we leave this to the interested reader.

\medskip

For $\sigma\in (0,\pi)$ let
\[\Sigma_\sigma := \{\lambda\in \C\setminus\{0\}: \arg(\lambda)<\sigma\}\]
denote the open sector of angle $\sigma$ in the complex plane.
A closed and densely defined operator $A$ on $X$ is {\em sectorial of type $\phi\in [0,\pi)$} if $A$ is one-to-one with dense range and for all $\sigma\in (\phi,\pi)$ we have $\Sigma_\sigma\subseteq \varrho(A)$ and
$$\sup_{\lambda\in \Sigma_{\sigma}} \|\lambda R(\lambda,A)\| <\infty.$$
Here, $R(\lambda,A) := (\lambda-A)^{-1}$.

We introduce the following condition on a sectorial operator $A$.
\begin{enumerate}
\item[$(H)$]\label{hyp:H} The operator $-A$ has a bounded $H^\infty$-calculus of angle $<\pi/2$.
\end{enumerate}
For details on $H^\infty$-calculus for sectorial operators we refer the reader to \cite{Haase:2,KuWe,McI,Weis-survey}.
The condition $(H)$ implies that $A$ generates an analytic semigroup $S(t) = e^{tA}$. Many differential operators on $L^q$-spaces with $q\in (1, \infty)$ which generate an analytic semigroup satisfy condition $(H)$.
We will show how one can use condition $(H)$ to obtain a continuity result for stochastic convolutions. Below we present several situations which are not covered by the existing literature. The existing results always require that the semigroup is contractive or quasi-contractive. Recall that $S(t)$ is called {\em quasi-contractive} if there exists a $w\in \R$ such that for all $t\geq 0$, $\|S(t)\|\leq e^{w t}$.

The next theorem is our first main result. It will be formulated for UMD Banach spaces $X$ with type $2$. Recall that $X=L^q$ with $q\in [2, \infty)$ is an example of a UMD space with type $2$. Moreover, every space which is isomorphic to a closed subspace of $L^q$ with $q\in (1, \infty)$ is UMD and of type $2$. For UMD spaces with type $2$ a class of stochastically integrable processes is given by the adapted and strongly measurable processes $G$ for which $G\in L^2(\R_+;\g(H,X))$ almost surely (see Proposition \ref{prop:stochinttheorytype2} for details). The set of all adapted $G$ which are in $L^2(\R_+;\g(H,X))$ is denoted by $L^0_{\F}(\O;L^2(\R_+;\g(H,X)))$.
For details on the space of $\gamma$-radonifying operators $\g(H,X)$ we refer to \cite{vNeeoverview}.

\begin{theorem}\label{thm:main1}
Let $X$ be a UMD space with type $2$. Assume $A$ satisfies hypothesis $(H)$. Then for all $G\in L^0_{\F}(\O;L^2(\R_+;\g(H,X)))$ the process $S\diamond G$ has a version with continuous paths. Moreover, for all $p\in (0,\infty)$, the following maximal estimate holds:
\begin{equation}\label{eq:maxest}
\big(\E\sup_{t\geq 0}\|S\diamond G(t)\|^p\big)^{1/p} \leq C_1 C_{2} (\E\|G\|_{L^2(\R_+;\g(H,X))}^p)^{1/p},
\end{equation}
where $C_1$ depends on $A$, and $C_2$ depends on $X$ and $p$.
\end{theorem}
As a corollary by an easy translation argument one can prove a version on bounded intervals $[0,T]$ in the case where only $A-w$ satisfies hypothesis $(H)$ for some $w>0$. This gives an extra exponential factor $e^{w T}$.

The proof of Theorem \ref{thm:main1} will be given in Section \ref{sec:proof} and uses a dilation argument in a similar spirit as \cite{HauSei1,HauSei2}. However, the enlargements of the spaces we need to consider are more complicated.
After the first version of this paper was written we found out that Theorem \ref{thm:main1} was also proved by Seidler in the setting $X=L^q$ with $q\in [2, \infty)$ (see \cite{Seidlernew}).

Under different conditions on the Banach space (see Section \ref{sec:alt}) it was proved in \cite{BrzPes} that for every $A$ for which $S(t)$ is (quasi)-contractive,
there exists a continuous version of $S\diamond G$.
Our result is not covered by this result since there are many interesting examples of differential operators $A$ which satisfy $(H)$, but for which $S(t)$ is not quasi-contractive, or not known to be quasi-contractive. For instance in \cite{LaMa} it was proved that semigroups generated by differential operators of order higher than two, are {\em never} contractive. Moreover, second order differential operators with irregular coefficients are often not quasi-contractive (see \cite[Theorems 1.1 and 1.2]{LSV}). Except for trivial cases, almost no positive results on quasi-contractiveness of semigroups generated by systems of differential operators are known. The only (easy to describe) class of scalar differential operators which generate an analytic semigroups which is quasi-contractive, seems to be second order differential operators in divergence form with smooth coefficient. Already in \cite{Amdual} it was shown that under fairly general boundary conditions these operators generate a quasi-contractive semigroup.

On the other hand, \cite{BrzPes} can be applied for instance to translation semigroups on $L^p$ with $p\in [2, \infty)$. This is not covered by Theorem \ref{thm:main1}, because condition $(H)$ implies analyticity of the semigroup.
In Section \ref{sec:alt} we present an alternative approach to obtain Theorem \ref{thm:main1} with slightly different assumptions on the Banach space $X$.

The following frequently arising examples in applications are not covered by the continuity theorem in \cite{BrzPes}.
\begin{example}\label{ex:secondnondiv}
Let $q\in (1, \infty)$. Let $A$ be a system of second order operators on a $C^2$-domain $\mathcal{O}\subset \R^n$:
\[(A f)(x) = \sum_{i,j=1}^n a_{ij}(x) D_i D_j f(s) + \sum_{i=1}^n b_i(x) D_i f(x) +c(x) f(x)\]
with Dirichlet boundary conditions. Let $X=L^q(\mathcal{O})$. If the $a_{ij}\in C^\epsilon(\overline{\mathcal{O}};\C^{N\times N})$ are uniformly parameter elliptic on $\overline{\mathcal{O}}$, and $b_i\in L^\infty(\mathcal{O};\C^{N\times N})$ and $c\in L^\infty(\mathcal{O};\C^{N\times N})$, then $A-w$ satisfies $(H)$ (see \cite{DDHPV}) for some $w\in R$ large enough. Therefore, if $q\in [2, \infty)$, then Theorem \ref{thm:main1} is applicable. However, $S(t)$ is not known to be quasi-contractive in this general situation.
\end{example}
The example can also be extended to systems of higher order elliptic operators as long as the Lopatinskii-Shapiro conditions hold (see \cite{DDHPV}).
As we already said before the semigroups generated by such higher order operators are never contractive.

In Section \ref{sec:alpha} we prove another result on the path-continuity of $S\diamond G$. Here we assume less on the Banach space $X$ and on the processes $G$. The result there covers all $L^q$-spaces with $q\in (1, \infty)$.

\section{Proof of Theorem \ref{thm:main1}\label{sec:proof}}

Before, we give the proof of the theorem recall the following result on stochastic integration theory. We refer to \cite{Brz2} and \cite{NVW1} for details.

\begin{proposition}\label{prop:stochinttheorytype2}
Assume $X$ is a UMD Banach space with type $2$. If $G:\R_+\times\O\to \g(H,X)$ is adapted and strongly measurable and $G\in L^2(\R_+;\g(H,X))$ a.s., then $G$ is stochastically integrable with respect to $W_H$, and the $X$-valued process $t\mapsto \int_0^t G(s) \, d W_H(s)$ is an a.s.\ pathwise continuous local martingale. Moreover, for all $p\in (0,\infty)$ one has
\begin{align}\label{eq:type2BDG}
\E\Big(\sup_{t\geq 0} \Big\|\int_0^t G(s) \, d W_H(s)\Big\|^p\Big) \leq C_{p,X}^p \|G\|_{L^p(\O;L^2(\R_+;\g(H,X)))}^p,
\end{align}
where $C_{p,X}$ depends only on $p$ and $X$.
\end{proposition}
The case $p\leq 1$ was not considered in \cite{Brz2} or \cite{NVW1}, but can easily be obtained by an application of Lenglart's inequality (see \cite{Lenglart}).

\begin{remark}\label{rem:optimal}
Recently, in \cite{Seidlernew} Seidler found the optimal asymptotic behavior of $C_{p,X}$ from \eqref{eq:type2BDG}. Using the result of \cite{Pin94} he showed that if $X$ is $2$-smooth, then there exists a constant $C_X$ such that $C_{p,X}\leq C_X \sqrt{p}$ for $p\geq 2$. This result applies to our setting in Proposition \ref{prop:stochinttheorytype2}, since UMD spaces with type $2$ have martingale type $2$ (see \cite{Brz1}) and such spaces can be renormed such that their norm is $2$-smooth (see \cite{Pi75}).
An alternative proof of the behavior of the constant in the setting of $L^q$-spaces based on interpolation, can be found in Corollary \ref{cor:BDGopt} below.
\end{remark}

\begin{lemma}\label{lem:UMDtype2}
For a Banach space $X$ the following assertions hold:
\begin{enumerate}
\item If $X$ has type $2$, then also the space $\gamma(L^2(\R_+;H),X)$ has type $2$.
\item If $X$ has UMD, then also the space $\gamma(L^2(\R_+;H),X)$ has UMD.
\end{enumerate}
\end{lemma}

\begin{proof}
It is well-known that the assertion holds if the space $\gamma(L^2(\R_+;H),X)$ is replaced by $L^2(\O;X)$. Now the result follows since $\gamma(L^2(\R_+;H),X)$ is isometric to a closed subspace of $L^2(\O;X)$, and therefore inherits the Banach spaces properties type $2$ and UMD.
\end{proof}

\begin{proof}[Proof of Theorem \ref{thm:main1}]
Let $Y = \gamma(L^2(\R), X)$. By \cite{FW} the boundedness of the $H^\infty$-calculus with angle $<\pi/2$, yields the following dilation result:

There are $J\in \calL(X, Y)$, $P\in \calL(Y)$ and $(U(t))_{t\in \R}$ in $\calL(Y)$ such that:
\begin{enumerate}
\item[(i)] There are $c, C\geq 0$ such that for all $x\in X$, one has $c\|x\| \leq \|J x\|_Y\leq C\|x\|$.
\item[(ii)] $P$ is a projection onto $J(X)$.
\item[(iii)] $(U(t))_{t\in \R}$ is a strongly continuous group on $Y$ with $\|U(t)y\|_Y = \|y\|_Y$ for all $y\in Y$.
\item[(iv)] For all $t\geq 0$ one has $J S(t) = P U(t) J$.
\end{enumerate}

Clearly, we have
\begin{align}
\label{eq:JSG}
J S\diamond G(t) = \int_0^t J S(t-s) G(s)\, d W_H(s) = P U(t) \int_0^t U(-s) J G(s)\, d W_H(s),
\end{align}
To see that the latter stochastic integral exists in $Y$, note that $s\mapsto U(-s) J G(s)$ is strongly measurable and adapted, and
\[\|U(-s) J G(s)\|_{L^2(\R_+;\g(H,Y))}\leq C\|G\|_{L^2(\R_+;\g(H,X))}<\infty \ \text{a.s.}\]
Since $Y$ has UMD and type $2$ by Lemma \ref{lem:UMDtype2}, it follows from Proposition \ref{prop:stochinttheorytype2} that $t\mapsto \int_0^t U(-s) J G(s)\, d W_H(s)$ exists and has a version which is a.s.\ pathwise continuous. Therefore, by \eqref{eq:JSG} and the strong continuity of $U(t)$ it follows that $J S\diamond G$ has a version which is a.s.\ pathwise continuous. By (i) also $S\diamond G$ has a version which is a.s.\ pathwise continuous. Moreover, if $G\in L^p(\O;L^2(\R_+;\g(H,X)))$ using \eqref{eq:JSG} and Proposition \ref{prop:stochinttheorytype2} one obtains the following estimate
\begin{align*}
\E \big(\sup_{t\geq 0}\|S\diamond G(t)\|^p\big) &\leq c^{-p} \E \big(\sup_{t\geq 0}\|J S\diamond G(t)\|^p_Y \big)
\\ & = c^{-p} \E\Big(\sup_{t\geq 0} \Big\|P U(t) \int_0^t U(-s) J G(s)\, d W_H(s)\Big\|^p_Y\Big)
\\ & \leq c^{-p} \|P\|^p \E\Big(\sup_{t\geq 0} \Big\|\int_0^t U(-s) J G(s)\, d W_H(s)\Big\|^p_Y\Big)
\\ & \leq c^{-p} \|P\|^p C_{p,Y}^p  \|U(-s) J G(s)\|_{L^p(\O;L^2(\R_+;\g(H,Y)))}^p
\\ & \leq c^{-p} \|P\|^p C_{p,Y}^p C^p \|G\|_{L^p(\O;L^2(\R_+;\g(H,X)))}^p.
\end{align*}
This completes the proof of \eqref{eq:maxest} with $C_1 = c^{-1} C \|P\|$ and $C_2 = C_{p,Y}$.
\end{proof}

It is natural to ask whether one also has exponential estimates (cf. \cite{BrzPes,HauSei2} and references therein) for $\sup_{t\geq 0} \|S\diamond G(t)\|$. This is indeed the case as follows from the next result.

\begin{theorem}\label{thm:exp1}
Assume $X$ is a UMD Banach space with type $2$. Assume condition $(H)$ holds. If $G\in L^0_{\F}(\O;L^2(\R_+;\g(H,X)))$ is such that
for some $M>0$, almost surely
\[\|G\|_{L^2(\R_+;\g(H,X))} \leq \sqrt{M},\]
then for every $R>0$,
\[\P(\sup_{t\geq 0} \|S\diamond G(t)\|\geq \lambda) \leq 2 \exp\Big(-\frac{\lambda^2}{2 e M C_1^2 B_q^2}\Big),\]
where $C_1$ is as in Theorem \ref{thm:main1} and $B_q$ only depends on $q$.
\end{theorem}
\begin{proof}
It follows from Remark \ref{rem:optimal} and Lemma \ref{lem:UMDtype2} that $C_{p,Y}\leq C_Y \sqrt{p}$ for $p\geq 2$. Therefore, we can conclude that $C_2$ from \eqref{eq:maxest} satisfies $C_2 = C_{p,Y}\leq C_Y \sqrt{p}$ for all $p\in [2, \infty)$.
Using power series argument as in \cite{HauSei2} it follows from \eqref{eq:maxest} that for any $\epsilon>0$
\begin{align*}
\E \exp(\epsilon \sup_{t\geq 0} \|S\diamond G(t)\|^2)  & = \sum_{n\geq 0} \E \sup_{t\geq 0} \frac{\epsilon^{n}\|S\diamond G(t)\|^{2n}}{n!}
 \leq \sum_{n\geq 0} C_1^{2n} C_2^{2n} M^{n} \frac{\epsilon^{n}}{n!}
\\ & = \sum_{n\geq 0} C_1^{2n} C_Y^{2n} (2n)^n M^{n} \frac{\epsilon^{n}}{n!}
 \leq \sum_{n\geq 0} C_1^{2n} C_Y^{2n} 2^n M^{n} \epsilon^{n} e^{n} =:I,
\end{align*}
where we used $n!\geq n^n e^{-n}$ in the last step. Clearly, the above expression $I = 2$ for $\epsilon = 2^{-1} e^{-1} M^{-1} C_1^{-2} C_Y^{-2}$.
The exponential estimate now follows from:
\begin{align*}
\P(\sup_{t\geq 0} \|S\diamond G(t)\|)\geq \lambda) &= \P\big(\exp(\epsilon\sup_{t\geq 0} \|S\diamond G(t)\|^2)\geq \exp(\epsilon \lambda^2)\big) \\ & \leq  e^{-\epsilon \lambda^2}\E \exp(\epsilon \sup_{t\geq 0} \|S\diamond G(t)\|^2)\leq 2 e^{-\epsilon \lambda^2}.
\end{align*}

\end{proof}
In the next section we present an entirely different approach to exponential tail estimates based on \cite{BrzPes}, which has the advantage that we do not need to have optimal estimates as $p \to \infty$. A disadvantage is that it is more difficult and the constants in the exponential estimate are less explicit.

\section{Alternative approach to Theorem \ref{thm:main1}\label{sec:alt}}

In this section we present an alternative approach to obtain a version of Theorem \ref{thm:main1} with a slightly different assumption on the geometry of the Banach space $X$ taken from \cite{BrzPes}. For $r\in [2, \infty)$ consider the following condition on $X$:

\let\ALTERWERTA\theenumi
\let\ALTERWERTB\labelenumi
\def\theenumi{(C$_r$)}
\def\labelenumi{\rm (C$_r$)}
\begin{enumerate}
\item\label{hyp:Cr} The function $\phi:X\to \R$ defined by $\phi(x) = \|x\|^r$ is two times continuously Fr\'echet differentiable and there are constants $k_1, k_2>0$ such that
\begin{equation}\label{eq:BrzPes}
\|\phi'(x)\|\leq k_1\|x\|^{r-1} \ \ \text{and} \ \ \|\phi''(x)\|\leq k_2 \|x\|^{r-2}.
\end{equation}
\end{enumerate}
\let\theenumi\ALTERWERTA
\let\labelenumi\ALTERWERTB

If \ref{hyp:Cr} holds for some $r\in [2, \infty)$, then one can show that (C$_{s}$) holds for all $s\geq r$. In particular, for $X=L^p$ with $p\in [2, \infty)$,  \ref{hyp:Cr} holds for all $r\in [p, \infty)$. Furthermore let us note that (C$_2$) can only hold for spaces which are isomorphic to a Hilbert space (see \cite[Fact 1.0 in V.I]{DGZ}). In particular, (C$_2$) does not hold for any $X = L^p$ with $p\in (2, \infty)$.

These estimates \eqref{eq:BrzPes} are the ones used in \cite{BrzPes} in order to obtain results on path-continuity under the additional assumption that $S(t)$ is a contraction semigroup.

The following result will allow us to relate our setting Theorem \ref{thm:main1} to the setting in \cite{BrzPes}.
\begin{proposition}\label{prop:Scontraction}
Let $X$ be a Banach space which satisfies \ref{hyp:Cr} for some $r\in [2, \infty)$. Assume $A$ satisfies hypothesis $(H)$. Then there exists an equivalent norm $\nnn\cdot\nnn$ on $X$ for which $X$ also satisfies \ref{hyp:Cr} with the same constants, and $S(t)$ is a contraction semigroup.
\end{proposition}
\begin{proof}
By hypothesis $(H)$ and \cite{KaWe} we can define the following equivalent norm on $X$:
\[\nnn x\nnn= \|t\mapsto (-A)^{1/2} S(t)x\|_{\g_r(\R_+;X)}.\]
Here for $r=2$, $\g_2(\R_+;X) = \g(\R_+;X)$ is as in \cite{NVW1, NVW3}. For $r\in (2, \infty)$, $\g_r(\R_+;X)$ is defined using $L^r(\O;X)$-norms of the corresponding Gaussian sums (which are all equivalent by the Kahane--Khintchine inequalities). Here $(\O, \mathcal{A}, \P)$ is a probability space.
We claim that $L^r(\O;X)$ satisfies \ref{hyp:Cr} with the same constants $k_1$ and $k_2$. Indeed, one can show that the function $\psi:X\to \R$ given by $\psi(y) = \|y\|_{L^r(\O;X)}^r$ satisfies
\begin{align*}
\psi'(y)v &= \int_\O \phi'(y(\omega)) u(\omega) \, d\P(\omega), \ \ u\in L^r(\O;X),
\\ \psi''(y)(u,v) &= \int_\O \phi''(y(\omega)) (u(\omega),v(\omega)) \, d\P(\omega), \ \ u,v\in L^r(\O;X).
\end{align*}
Now the claim follows from H\"older's inequality and the assumption on $X$.

As $\g_r(\R_+;X)$ is a closed subspace of $L^r(\O;X)$, it also satisfies \ref{hyp:Cr} with constants $k_1$ and $k_2$. Finally, $S$ is a contraction semigroup in $(X, \nnn\cdot\nnn)$, since
\begin{align*}
\nnn S(s)x\nnn & = \|t\mapsto (-A)^{1/2} S(t+s) x\|_{\g_r(\R_+;X)}
\\ & = \|t\mapsto (-A)^{1/2} S(t) x\|_{\g_r([s, \infty);X)}\leq \nnn x\nnn,
\end{align*}
where we used the left ideal property in $\gamma_r(\R_+;X)$ in the last line.
\end{proof}

As a consequence we obtain the following result.
\begin{theorem}\label{thm:main3}
Let $X$ be a Banach space with satisfies condition \ref{hyp:Cr}. Assume condition $(H)$ holds. Then for all $G\in L^0_{\F}(\O;L^2(\R_+;\g(H,X)))$ the process $S\diamond G$ has a version with continuous paths. Moreover, for all $p\in (0,\infty)$, the following maximal estimate holds:
\begin{equation}\label{eq:maxest3}
\big(\E\sup_{t\geq 0}\|S\diamond G(t)\|^p\big)^{1/p} \leq C_1 C_{2} (\E\|G\|_{L^2(\R_+;\g(H,X))}^p)^{1/p},
\end{equation}
where $C_1$ depends on $A$, and $C_2$ depends on $X$ and $p$.
\end{theorem}

\begin{proof}
First let $p\geq r$. Then $X$ satisfies (C$_p$). By Proposition \ref{prop:Scontraction} we can find an equivalent norm $\nnn\cdot\nnn$ on $X$ which satisfies (C$_p$) and for which $S$ is a contraction semigroup. Let $b, B>0$ be such that $b\nnn x\nnn \leq \|x\|\leq B\nnn x\nnn$. By \cite[Theorem 1.1]{BrzPes} we obtain a version with continuous paths. Moreover, by \cite[(1.2)]{BrzPes} we can find a constant $K$ depending on the constants in \eqref{eq:BrzPes} and $p$ such that
\begin{align*}
\big(\E\sup_{t\geq 0}\|S\diamond G(t)\|^p\big)^{1/p} &\leq B\big(\E\sup_{t\geq 0}\nnn S\diamond G(t)\nnn ^p\big)^{1/p}
\\ & \leq B K  (\E\nnn G\nnn_{L^2(\R_+;\g(H,X))}^p)^{1/p} \\ & \leq B K b^{-1}  (\E\|G\|_{L^2(\R_+;\g(H,X))}^p)^{1/p}
\end{align*}
This proves the result with $C_1 = B b^{-1}$ and $C_2 = K$ for $p\geq r$. For $0<p<r$, the result follows from a standard application of Lenglart's stopping time argument (see \cite{Lenglart}).
\end{proof}

One drawback of the above approach is that the constant $C_2$ that comes from the proof of \cite[Theorem 1.1]{BrzPes} is somewhat complicated and probably not optimal.
On the other hand by \cite[Theorem 1.2]{BrzPes} we immediately get exponential estimates.
\begin{theorem}\label{thm:exptail}
Let $X$ be a Banach space with satisfies condition \ref{hyp:Cr}. Assume condition $(H)$ holds. If $G\in L^0_{\F}(\O;L^2(\R_+;\g(H,X)))$ is such that
for some $M>0$, almost surely
\[\|G\|_{L^2(\R_+;\g(H,X))} \leq \sqrt{M},\]
then for every $R>0$,
\[\P(\sup_{t\geq 0} \|S\diamond G(t)\|\geq R) \leq 3 \exp\Big(-\frac{R^2}{C_1 C_2 M}\Big).\]
where $C_1$ depends on $A$ and $C_2$ depends on $X$.
\end{theorem}
\begin{proof}
Let $\nnn\cdot\nnn$ be as in the proof of Theorem \ref{thm:main3}. We write $\tilde{X}$ for $X$ with the norm $\nnn\cdot\nnn$. Then
\[\|G\|_{L^2(\R_+;\g(H,\tilde{X}))}^2 \leq b^{-2}\|G\|_{L^2(\R_+;\g(H,X))}^2 \leq b^{-2} M.\]
Therefore, \cite[Theorem 1.2]{BrzPes} implies that there is a constant $K>0$ depending on $p$ and $X$ such that
\begin{align*}
\P(\sup_{t\geq 0} \|S\diamond G\|\geq R) &\leq \P(\sup_{t\geq 0} \nnn S\diamond G(t)\nnn\geq R/B)
\\ & \leq 3 \exp\Big(-\frac{R^2}{B^2 b^{-2} K M}\Big).
\end{align*}
The result follows with $C_1 = B^2 b^{-2}$ and $C_2 = K$.
\end{proof}

\section{Extensions of the results for spaces with property $(\alpha)$\label{sec:alpha}}

In this section we present a result which does not require the type $2$ assumption on the Banach space $X$. However, we do assume $X$ is a UMD space. In this setting the space of integrable processes is described by the space $L^0_{\F}(\O;\g(\R_+;H,X))$ and one has the following (see \cite{NVW1} for details):
\begin{proposition}\cite[Theorems 5.9, 5.12]{NVW1}\label{prop:NVW}
Let $E$ be a UMD Banach space and let $p\in (0,\infty)$ be fixed. For an adapted process $\Phi:\R_+\times\O\to \calL(H,X)$ the following are equivalent:
\begin{enumerate}
\item The process $\Phi$ is stochastically integrable with respect to $W_H$.
\item $\Phi(\cdot, \omega)\in \g(\R_+;H,X)$ for a.e.\ $\omega\in\Omega$.
\end{enumerate}

In this situation we have $t\mapsto \int_0^{t} \Phi\,dW_H$ is a.s.\ pathwise continuous. Furthermore, for all $p\in (0, \infty)$, there exists constants $c_{p,X}^{\g},C_{p,X}^{\g}>0$ such that
\[
c_{p,X}^{\g}\E\n \Phi\n_{\g(0,T;H,X)}^p \leq  \E \sup_{t\in [0,T]}\Big\n \int_0^t \Phi\,dW_H\Big\n^p \leq C_{p,X}^{\g}
\E\n \Phi\n_{\g(0,T;H,X)}^p.
\]
\end{proposition}
The case $0<p\leq 1$ was not considered in \cite{NVW1}, but can easily be obtained by an application of Lenglart's inequality.

In the following result we need that the Banach space $X$ has the so-called property $(\alpha)$ (see \cite{Pialpha} for details). Examples of UMD spaces with property $(\alpha)$ are $X=L^q$ with $q\in (1, \infty)$ or any space which is isomorphic to a closed subspace of $L^q$ with $q\in (1, \infty)$.

\begin{theorem}\label{thm:main2}
Let $X$ be a UMD space with property $(\alpha)$. Assume $A$ satisfies hypothesis $(H)$. Then for all $G\in L^0_{\F}(\O;\g(\R_+;H,X))$ the process $S\diamond G$ has a version with continuous paths. Moreover, for all $p\in (0,\infty)$, the following maximal estimate holds:
\begin{equation}\label{eq:maxest2}
\big(\E\sup_{t\geq 0}\|S\diamond G(t)\|^p\big)^{1/p} \leq C_1 C_{2} (\E\|G\|_{\gamma(\R_+;H,X)}^p)^{1/p},
\end{equation}
where $C_1$ depends on $A$ and $w$, and $C_2$ depends on $X$ and $p$.
\end{theorem}

In the case that $X$ has type $2$ and property $(\alpha)$ the assertion in Theorem \ref{thm:main2} is stronger than Theorem \ref{thm:main1}. Indeed, this follows from the fact that for type $2$ spaces $X$ the space $\gamma(\R_+;H,X)$ is larger than $L^2(\R_+;\g(H,X))$ (see \cite{vNWe,RS}).

Theorem \ref{thm:main2} applies to the same situation as in Example \ref{ex:secondnondiv}.

\begin{example}
Let $q\in (1, \infty)$. Let $A$ and $X$ be as in Example \ref{ex:secondnondiv}. Then as before $A-w$ satisfies $(H)$ for some $w\in R$ large enough. Therefore, Theorem \ref{thm:main2} is applicable for any $q\in (1, \infty)$. Moreover, even for $q\in [2, \infty)$ the assertion of Theorem \ref{thm:main2} leads to stronger results in this example.
\end{example}

The proof of Theorem \ref{thm:main2} is more involved. We need to apply property $(\alpha)$ to have better structural properties of the group used in the dilation argument.

\begin{proof}[Proof of Theorem \ref{thm:main2}]
Let $Y$, $J\in \calL(X, Y)$, $P\in \calL(Y)$ and $(U(t))_{t\in \R}$ in $\calL(Y)$ be as in the proof of Theorem \ref{thm:main1}.
The equality \eqref{eq:JSG} still holds.
However, we need some arguments to see that the stochastic integral $\int_0^t U(-s) J G(s)\, d W_H(s)$ exists in $Y$.
Indeed, note that $s\mapsto U(-s) J G(s)$ is strongly measurable and adapted. Recall from \cite{FW} that $U(r)\in \calL(Y)$ is the tensor extension (in the sense of \cite{KaWe}) of the usual right-translation operator on $L^2(\R;H)$. Since $X$ has property $(\alpha)$, it follows from \cite[Theorem 3.18]{HaKu} that $(U(r))_{r\in \R}\subseteq \calL(Y)$ is $\g$-bounded by some constant $\alpha_X$. Now the multiplier result of \cite{KaWe} shows that $s\mapsto U(-s) J G(s)$ is in $\g(\R_+;H;X)$ a.s., and
\begin{equation}\label{eq:alpha}
\begin{aligned}
\|U(-s) J G(s)\|_{\g(\R_+;H,Y)}& \leq \alpha_X \|J G(s)\|_{\g(\R_+;H,Y)} \\ & \leq \alpha_X C\|G\|_{\g(\R_+;H;X)}<\infty \ \text{a.s.},
\end{aligned}
\end{equation}
where the last step follows from the left-ideal property.
Since $Y$ has UMD by Lemma \ref{lem:UMDtype2}, it follows from Proposition \ref{prop:NVW} that $t\mapsto \int_0^t U(-s) J G(s)\, d W_H(s)$ exists and has a version which is a.s.\ pathwise continuous. Therefore, by \eqref{eq:JSG} and the strong continuity of $U(t)$ it follows that $J S\diamond G$ has a version which is a.s.\ pathwise continuous. As in the proof of Theorem \ref{thm:main1} by (i) also $S\diamond G$ has a version which is a.s.\ pathwise continuous. Moreover, if $G\in L^p(\O;\g(\R_+;H,X))$ we can use \eqref{eq:JSG}, Proposition \ref{prop:NVW} and \eqref{eq:alpha} to obtain the following estimate
\begin{align*}
\E \big(\sup_{t\geq 0}\|S\diamond G(t)\|^p\big) &\leq c^{-p} \E \big(\sup_{t\geq 0}\|J S\diamond G(t)\|^p_Y \big)
\\ & = c^{-p} \E\Big(\sup_{t\geq 0} \Big\|P U(t) \int_0^t U(-s) J G(s)\, d W_H(s)\Big\|^p_Y\Big)
\\ & \leq c^{-p} \|P\|^p \E\Big(\sup_{t\geq 0} \Big\|\int_0^t U(-s) J G(s)\, d W_H(s)\Big\|^p_Y\Big)
\\ & \leq c^{-p} \|P\|^p (C_{p,Y}^{\g})^p  \|U(-s) J G(s)\|_{L^p(\O;\g(\R_+;H,X))}^p
\\ & \leq c^{-p} \|P\|^p (C_{p,Y}^{\g})^p \alpha_X^p C^p \|G\|_{L^p(\O;\g(\R_+;H,X))}^p.
\end{align*}
This completes the proof of \eqref{eq:maxest} with $C_1 = c^{-1} C \|P\|$ and $C_2 = C_{p,Y}^{\g} \alpha_X$.
\end{proof}

At this moment we do not know if there are exponential tail estimates in the general setting of Theorem \ref{thm:main2}. However, also in this setting there is some hope that $C_{p,Y}^{\g}\leq C_{X} \sqrt{p}$ for $p$ large, and by the argument in Theorem \ref{thm:exp1} this would yield exponential tail estimates again. Recently, in \cite{CoxVer2} it has been proved that $C_{p,Y}^{\g}\leq C_{X} p$ for $p$ large. This yields exponential estimates, but no exponential quadratic estimates as one would expect.

\appendix
\section{Optimal constants in the Burkholder-Davis-Gundy inequality for stochastic integrals}

For a Banach space $X$ and $p\in (0,\infty)$ let $K_{p,X}$ be the smallest constant $K$ such that
\begin{align}\label{eq:BDGwithoutsup}
\sup_{t\geq 0} \E\Big( \Big\|\int_0^t G(s) \, d W_H(s)\Big\|^p\Big) \leq K^p \|G\|_{L^p(\O;L^2(\R_+;\g(H,X)))}^p,
\end{align}
for all $G\in L^p_{\F}(\O;L^2(\R_+;\g(H,X)))$. If there does not exist such a constant $K$, we set $K_{p, X}=\infty$. Recall from Proposition \ref{prop:stochinttheorytype2} that \eqref{eq:BDGwithoutsup} holds for some $K$ if $X$ is a UMD space with type $2$. Moreover, in \cite{Seidlernew} it has been proved that $K_{p,X}\leq K_X \sqrt{p}$ for $p\geq 2$ (also see Remark \ref{rem:optimal}). Below we provide an alternative proof of this fact for the case $X = L^q$ with $q\in [2, \infty)$.
Also recall from real stochastic analysis that there is a constant $b>0$ such that for all $p\geq 2$, $K_{p,\R}\leq b \sqrt{p}$ (see \cite{Dav76}).

\begin{proposition}\label{prop:interpstoch}
Let $X_0$ and $X_1$ be Banach spaces for which \eqref{eq:BDGwithoutsup} holds and which form an interpolation couple. Assume $X_0$ is reflexive. Then the complex interpolation spaces $X_\theta = [X_0, X_1]_{\theta}$ with $\theta\in [0,1]$ satisfies \eqref{eq:BDGwithoutsup} with
\[K_{p,X_{\theta}}\leq  K_{p,X_{0}}^{1-\theta} K_{p,X_{1}}^{\theta}.\]
\end{proposition}
\begin{proof}
One easily checks that \eqref{eq:BDGwithoutsup} implies that $X_1$ and $X_2$ have type $2$ (see \cite{RS}). Therefore, $X_1$ and $X_2$ are $K$-convex and this implies $[\g(H,X_0), \g(H, X_1)]_{\theta} = \g(H,X_\theta)$ (see \cite[Proposition 2.3]{HvNP} or \cite{SuWe}))

Fix $t\in \R_+$ and let $Y = L^p_{\F}(\O;L^2(0,t))$. Clearly, $Y$ is Banach function space. As in \cite{Cal} write $Y(X)$ for the $X$-valued strongly measurable and adapted processes $g$ with values in $X$ for which $\|g\|_{Y(X)} := \| \, \|g\|_X\, \|_Y<\infty$. We claim that $Y(\g(H, X_0))$ is reflexive. Indeed, note that $\g(H,X_0)$  is isometric to a closed subspace of $L^2(\tilde{\O};X_0)$ for some probability space $(\tilde{\O}, \tilde{\F}, \tilde{\P})$, and the latter is reflexive since $X_0$ is reflexive. Therefore, $\g(H, X_0)$ is reflexive as well.
Now the claim follows from the fact that $Y(\g(H, X_0))$ is a closed subspace of the reflexive space $L^p(\O;L^2(0,t;\g(H,X_0)))$. By \cite[13.5]{Cal} we obtain
\[[Y(\g(H,X_0)), Y(\g(H,X_1))]_{\theta}= Y([\g(H,X_0), \g(H, X_1)]_{\theta}) = Y(\g(H,X_\theta)).\]
Similarly, one has
\[[L^p(\O;X_{0}),L^p(\O;X_1)]_\theta = L^p(\O;X_{\theta}).\]

Let $T:Y(\g(H,X_i))\to L^p(\O;X_i)$ be defined by
$T G= \int_0^t G(s) \, d W_H(s)$.
Then $\|T\|_{Y(\g(H,X_i))\to L^p(\O;X_i)}\leq K_{p,X_i}$ for $i=1,2$.
Consequently, since $[\cdot, \cdot]_{\theta}$ is an (exact) interpolation method (see \cite[Theorem 1.9.3]{Tr1}), we obtain
\begin{align*}
\|T&\|_{Y(\g(H,X_\theta))\to L^p(\O;X_\theta)}  \leq K_{p,X_0}^{1-\theta} K_{p,X_1}^{\theta}.
\end{align*}
Since $t\in \R_+$ was arbitrary, we obtain $K_{p,X_\theta} \leq K_{p,X_0}^{1-\theta} K_{p,X_1}^{\theta}$.
\end{proof}

\begin{lemma}\label{lem:easycases}
Let $p\in [2, \infty)$. The following assertions hold:
\begin{enumerate}
\item If $(\mathcal{O}, \Sigma, \mu)$ be a (nonempty) $\sigma$-finite measure space, then $K_{p,L^p(\mathcal{O})} = K_{p,\R}$.
\item If $X$ is a Hilbert space with nonzero dimension, then $K_{p,X} = K_{p,\R}$.
\end{enumerate}
\end{lemma}

\begin{proof}
\

(1):\ This follows from Fubini theorem. However, due to operator valued setting some technicalities have to be overcome.
Write $X=L^p(\mathcal{O})$. By a density argument, it suffices to consider adapted step processes $G$ which take values in the finite rank operators, i.e.
\[ G = \sum_{n=1}^N \one_{(t_{n-1},t_n]} \sum_{m=1}^M \one_{A_{mn}} \sum_{j=1}^J h_{j}\otimes x_{j mn}.\]
Here $0=t_0< t_1<\ldots<t_{N}=t$, the sets $(A_{mn})_{m=1}^M$ are in $\F_{t_n}$, $(h_{j})_{j=1}^J$ in $H$ are orthonormal and $(x_{jmn})_{j,m,n}$ are in $X$.

Let $g:\R_+\times\O\times \mathcal{O}\to H$ be given by
\[g =  \sum_{n=1}^N \one_{(t_{n-1},t_n]} \sum_{m=1}^M \one_{A_{mn}} \sum_{j=1}^J x_{j mn} \otimes h_{j}.\]
Now fix some time $t>0$.
Recall that $\g(H,\R) = H$ By Fubini's theorem we can write
\begin{align*}
[\E\Big( \Big\|\int_0^t G(s) \, d W_H(s)\Big\|^p_{X} &= \int_{\mathcal{O}} \E\Big|\int_0^t g(s,\cdot, r) \, d W_H(s)\Big|^p \, d\mu(r)
\\ & \leq K_{p,\R}^p \int_{\mathcal{O}}  \|g(\cdot,\cdot, r)\|_{L^p(\O;L^2(0,t;H))}^p  \, d\mu(r)
\\ & = K_{p,\R}^p \|g\|_{L^p(\O;L^p(\mathcal{O}; L^2(0,t;H)))}^p
\\ & \stackrel{(i)}{\leq} K_{p,\R}^p \|g\|_{L^p(\O; L^2(0,t;L^p(\mathcal{O};H)))}^p
\\ & \stackrel{(ii)}{\leq} \|G\|_{L^p(\O;L^2(\R_+;\g(H,X)))},
\end{align*}
The estimate (i) follows from Minkowski's inequality with exponent $p/2$. To see that (ii) holds, let $f\in L^p(\mathcal{O};H)$ and $F\in \gamma(H, L^p(\mathcal{O}))$ be given by $(F h)(r) = [h,f(r)]_H$. Let $(h_j)_{j\geq 1}$ be an orthonormal basis for $H$. Then by randomization and Minkowski's inequality with exponent $p/2$, we have
\begin{align*}
\|f\|_{L^p(\mathcal{O};H)} &= \Big\|\Big(\sum_{j\geq 1} |[h_j,f]|^2\Big)^{1/2} \Big\|_{L^p(\mathcal{O})}
 = \Big\|\Big(\sum_{j\geq 1} |F h_j|^2\Big)^{1/2} \Big\|_{L^p(\mathcal{O})}
\\ & = \Big\|\sum_{j\geq 1} \gamma_j F h_j\|_{L^p(\mathcal{O};L^2(\widetilde{\O}))}
 \leq \Big\|\sum_{j\geq 1} \gamma_j F h_j\|_{L^2(\widetilde{\O};L^p(\mathcal{O}))}
 = \|F\|_{\g(H,L^p(\mathcal{O}))}.
\end{align*}
Here $(\gamma_j)_{j\geq 1}$ is a Gaussian sequence on a probability space $(\widetilde{\O}, \widetilde{\mathcal{A}}, \widetilde{\P})$.
This proves (ii) and therefore, $K_{p,X}\leq K_{p,\R}$. The converse estimate is trivial.

(2): This seems to be well-known to experts. A short proof can be given using (1). Fix $G\in L^p_{\mathcal{F}}(\O;L^2(\R_+;\g(H,X))$. Since $G$ is strongly measurable it takes values in a separable subspace of $\g(H,X)$. Thereforem we can replace $X$ by a separable Hilbert space $X_0$ if necessary. Now the result follows from (1), because any separable Hilbert space is isometric to a closed subspace of $L^p(0,1)$ (see \cite[Proposition 6.4.13]{AlKa}).
\end{proof}

As a consequence we obtain the following result.

\begin{theorem}\label{thm:constantLq}
Let $(\mathcal{O}, \Sigma, \mu)$ be a (nonempty) $\sigma$-finite measure space and let $q\in [2, \infty)$. Let $X$ be a closed subspace of $L^q(\mathcal{O})$. Then for all $p\in [q, \infty)$ one has $K_{p,X} = K_{p,\R}$ for the optimal constants from \eqref{eq:BDGwithoutsup}.
\end{theorem}
\begin{proof}
Without loss of generality we can assume $X = L^q(\mathcal{O})$. Let $\theta\in (0,1)$ be such that $\frac1q = \frac{1-\theta}{2} + \frac{\theta}{p}$. Then it follows from Proposition \ref{prop:interpstoch} with $X_0 = L^2(\mathcal{O})$ and $X_1 = L^p(\mathcal{O})$ and $X_{\theta} = L^q(\mathcal{O})$ that
$K_{p,L^q(\mathcal{O}} \leq K_{p,L^2(\mathcal{O})}^{1-\theta} K_{p,L^p(\mathcal{O})}^{\theta}$.
Combining this with Lemma \ref{lem:easycases} yields $K_{p,L^q(\mathcal{O})}\leq K_{p,\R}$. The converse inequality is trivial.
\end{proof}

\begin{corollary}\label{cor:BDGopt}
Let $(\mathcal{O}, \Sigma, \mu)$ be a (nonempty) $\sigma$-finite measure space, let $q\in [2, \infty)$ and let $X$ be a closed subspace of $L^q(\mathcal{O})$. Then for all $p\in [q, \infty)$,
\[\Big(\E\sup_{t\geq 0} \Big\|\int_0^t G(s) \, d W_H(s)\Big\|_{X}^p\Big)^{1/p} \leq K_{p,\R} p' \|G\|_{L^p(\O;L^2(\R_+;\g(H,X)))},\]
where $p'\in (1, 2]$ is such that $\frac{1}{p}+\frac{1}{p'}=1$.
\end{corollary}
Recall that there is a constant $b>0$ such that $K_{p,\R}\leq b \sqrt{p}$ for all $p\in [2, \infty)$. Therefore, in the above result we have $K_{p,\R} \, p'\leq 2 b \sqrt{p}$ as soon as $p\in [q,\infty)$. This is a rather precise description of the behavior of the constant as $p\to \infty$ and has important consequences.

\begin{proof}
This follows directly from Doob's maximal $L^p$-inequality for the submartingale  $\Big\|\int_0^\cdot G(s) \, d W_H(s)\Big\|$ combined with \eqref{thm:constantLq}.
\end{proof}

\medskip

\noindent {\em Acknowledgement} \
In the first version of this paper the proof of Theorem \ref{thm:exp1} was only valid for the case $X = L^q$. The authors thank professor Jan Seidler for sending them his paper \cite{Seidlernew} where it has been proved that the constants in Proposition \ref{prop:stochinttheorytype2} behave as $C_{p,X}\leq C_X \sqrt{p}$ for $p$ large.
This enabled us to cover all UMD spaces with type $2$ in Theorem \ref{thm:exp1}.

The authors would like to thank the anonymous referee for pointing out a mistake in Section \ref{sec:alt}.

\def\cprime{$'$} \def\cprime{$'$} \def\cprime{$'$}

\end{document}